\documentclass[reqno,twoside,10pt,english]{amsart}
\usepackage{amsmath,amsfonts,amssymb,amsthm}
\usepackage[pagebackref,colorlinks,citecolor=blue,linkcolor=blue]{hyperref}
\usepackage{color}
\usepackage{graphicx}
\allowdisplaybreaks

\let\th=T

\def\XXint#1#2#3{{\setbox0=\hbox{$#1{#2#3}{\int}$ }
		\vcenter{\hbox{$#2#3$ }}\kern-.6\wd0}}

\newtheorem{theorem}{Theorem}[section]

\newtheorem{proposition}[theorem]{Proposition}

\begin{document}

\title{Recovery of a source of anomalous  diffusion from space-time sparse data$^*$}

\author{Xing Cheng$^1$, Bangti Jin$^2$, Yavar Kian$^3$, and Zhiyuan Li$^4$}

\begin{abstract}
We study an inverse source problem for a time-fractional sub- and super-diffusive anomalous diffusion equation on the unit disk with sparse boundary observations. We prove the recovery of a general class of $L^2$ sources for anomalous diffusion, with separated variables, from boundary measurements that are sparse in  both space and time, using observations at two carefully chosen boundary points on a discrete set of times arbitrarily close to the final time. The analysis relies heavily on the memory effect of the model, properties of complex analysis and the representation of solutions.

\bigskip

\noindent \textbf{Keywords}: fractional diffusion equation, inverse source problem, sparse measurement, uniqueness

\bigskip

\noindent \textbf{Mathematics Subject Classification (2020)} Primary: 35R30; Secondary: 35R11, 26A33
\end{abstract}

\maketitle

\renewcommand{\thefootnote}{\fnsymbol{footnote}}
\footnotetext{\hspace*{-5mm} 
\begin{tabular}{@{}r@{}p{16cm}@{}}
$^*$
&X. Cheng has been supported by the NSF of Jiangsu Province (No.~BK20230036), the NSF of China (Grant U2340221) and the National Key R\&D Program of China (2022YFC3202602).  The work of B. Jin is supported by Hong Kong RGC ANR / Hong Kong RGC Joint Research Scheme (A-CUHK402/24). The work of Y. Kian is supported by the French National Research Agency ANR and Hong Kong RGC Joint Research Scheme for the project IdiAnoDiff (grant ANR-24-CE40-7039). Z. Li has been supported by the National Natural Science Foundation of China (no. 12271277), the Ningbo Youth Leading Talent Project (no. 2024QL045), and the Open Research Fund of the Key Laboratory of Nonlinear Analysis \& Applications (Central China Normal University), Ministry of Education, China (no. NAA20230RG002).\\
$^1$ 
& School of Mathematics, Hohai University, Nanjing 210098, Jiangsu, China, and Laboratory of Mathematical Modeling and Intelligent Computing for Water Systems, Hohai University, Nanjing 211100, China.(\texttt{chengx@hhu.edu.cn}).\\
$^2$
& Department of Mathematics, The Chinese University of Hong Kong, Shatin, N.T., Hong Kong, China (\texttt{b.jin@cuhk.edu.hk}).\\
$^3$
& Univ Rouen Normandie, CNRS, Normandie Univ, LMRS UMR 6085, F-76000 Rouen, France  ( \texttt{yavar.kian@univ-rouen.fr}).\\
$^4$
&School of Mathematics and Statistics, Ningbo University, Ningbo 315211, China ( \texttt{lizhiyuan@nbu.edu.cn}).
\end{tabular}}

\section{Introduction}
Throughout, fix $\alpha\in(0,1)\cup(1,2)$, $T>0$, and let $D:=\{x\in\mathbb R^2: |x|<1\}$ be the unit disk, with boundary $\partial D=\mathbb S^1:=\{x\in\mathbb R^2: |x|=1\}$.  For $\sigma\in L^2(0,T)$ and $f\in H^s(D)$,   $s\in(0,\frac12)$,
consider the following initial boundary value problem (IBVP)
\begin{equation}\label{eq-gov}
\begin{cases}
  \partial_t^\alpha u - \Delta u =F(x,t):= \sigma(t)f(x), & (x,t) \in D \times (0,T), \\
  \partial_t^k u(x,0) = 0,\quad k=0,\cdots,\lfloor \alpha \rfloor, & x \in D,\\
  u(x,t) = 0, & (x,t) \in \partial D \times (0,T),
\end{cases}
\end{equation}
where $\partial_t^\alpha$ denotes the Caputo derivative of order $\alpha$ in time:
$$
\partial_t^\alpha \psi(t) := \frac{1}{\Gamma(1+\lfloor \alpha \rfloor-\alpha)} \int_0^t (t-\tau)^{\lfloor\alpha\rfloor-\alpha} \psi^{(\lfloor\alpha\rfloor+1)}(\tau) \,\mathrm{d}\tau, \quad t>0.
$$
Below we make the standing assumption that there exists $T_1\in (0,T)$ such that $\sigma=0$ on $(T_1,T)$. Using this property, we prove in Proposition \ref{p1} that problem \eqref{eq-gov} admits a unique solution $u\in L^2(0,T;H^2(D))$ such that $u|_{\overline{D}\times (T_1,T]}\in C((T_1,T];C^1(\overline{D}))$. Then, by fixing an arbitrarily chosen strictly increasing sequence $(t_n)_{n\in\mathbb N}$ in $(T_1,T)$ and $z_1,z_2\in\mathbb S^1$, we investigate the following inverse source problem:
\begin{itemize}
\item[{\bf(ISP)}] {\em
Determine uniquely the source $F$ from the space-time sparse boundary data $\partial_\nu u(x_\ell,t_n)$}, $n\in\mathbb N$,
$\ell=1,2$.
\end{itemize}

Anomalous diffusion models of the form \eqref{eq-gov} serve as canonical models for anomalous transport processes arising in heterogeneous porous media, viscoelastic materials, and disordered systems \cite{metzler2000random}. For $\alpha\in(0,1)$, the model characterizes subdiffusion with a sublinearly growing mean squared displacement, and for $\alpha\in(1,2)$, it describes the diffusion-wave regime. These models have been successfully employed in various fields, including dispersion in heterogeneous aquifers \cite{AdamsGelhar:1992} and ion dispersion in column experiments \cite{HatanoHatano:1998} (see also the review \cite{metzler2000random}). In this context, the inverse source problem \textbf{(ISP)} aims at recovering the source $F$ of anomalous diffusion from sparse boundary measurements of the flux.

Inverse source problems for time-fractional diffusion have attracted considerable attention from the mathematical community (see the surveys \cite{jin2015tutorial, liu2019handbook}). We only mention the early work \cite{sakamoto2011initial} as well as more recent contributions where the memory effect of fractional diffusion equations has been exploited to restrict measurements close to the final time \cite{JaK,KLY}.

Despite these important progresses, the above-mentioned results require measurements that are sometimes inconvenient to implement in practical applications. Motivated by such limitations, several researchers have investigated the recovery of a diffusion source in the two-dimensional unit disc $D$ by restricting measurements to two points on the boundary $\partial D$ along the time interval $(0,T)$. This includes results for diffusion type equations \cite{DCK,gong2026identification,GZZ,HR,Run2020Heatequation_source_problem} as well as their extension by \cite{li2020unique} to a broader class of time-fractional diffusion problems. These contributions significantly restrict the amount of the data. Nonetheless, all these works still involve measurements over a full continuous time interval. This limits the range of such results by preventing engineering applications where measurements are usually available only at a small number of spatial sensors and discrete time instants due to experimental limitations and cost constraints. Motivated by these limitations of the existing literature, the goal of this article is to provide a positive answer to problem \textbf{(ISP)} for a certain class of time-fractional diffusion equations.

First, we provide a positive answer to problem \textbf{(ISP)} when the time-dependent factor $\sigma$ is known.

\begin{theorem}\label{thm-spatial}
Let $\sigma\in L^2(0,T)$ with $\sigma\not\equiv0$, and assume that $\sigma=0$ on $(T_1,T)$ for some $T_1\in(0,T)$. Let $f_j\in H^s(D)$, $j=1,2$, $s\in(0,\frac12)$, let $z_j=e^{i\theta_j}\in\mathbb S^1$ with $\theta_1-\theta_2\notin\mathbb Q\pi$, and let $(t_n)_{n\in\mathbb N}$ be a strictly increasing sequence in $(T_1,T)$. Let $u_j$ be the solution to \eqref{eq-gov} with source $\sigma f_j$, $j=1,2$. If
\begin{equation}\label{eq-obs1}
\partial_\nu u_1(z_\ell,t_n)=\partial_\nu u_2(z_\ell,t_n),\quad n\in\mathbb N,\ \ell=1,2,
\end{equation}
then $f_1=f_2$ in $D$.
\end{theorem}

To the best of our knowledge, Theorem \ref{thm-spatial} provides the first positive answer to problem \textbf{(ISP)}. In contrast to  \cite{li2020unique}, the measurements can be associated with a discrete subset of $\partial D\times (T-\epsilon,T)$ with $\epsilon>0$ arbitrarily small. This makes the result of Theorem \ref{thm-spatial} much more suitable for applications that can only use discrete data. In addition, Theorem \ref{thm-spatial} agrees with recent developments in inverse source problems for fractional diffusion equations involving a posteriori measurements \cite{JaK,KLY}.
The proof of Theorem \ref{thm-spatial} relies fundamentally on the memory property of \eqref{eq-gov}, explicit spectral properties of the Dirichlet Laplacian on the unit disc and representation formulas for solutions of \eqref{eq-gov} using complex analysis arguments.

While in general it is impossible to provide a full answer to problem \textbf{(ISP)} when $\sigma$ is unknown, we can still give the following positive answer when the fractional order $\alpha$ is irrational.

\begin{theorem}\label{thm-full}
Let $\alpha\not\in\mathbb Q$, $\sigma_j\in L^2(0,T)$ with $\sigma_j\not\equiv0$, $j=1,2$, and assume that $\sigma_1=\sigma_2=0$ on $(T_1,T)$ for some $T_1\in(0,T)$. Let $f_j\in H^s(D)$, $j=1,2$, $s\in(0,\frac12)$, let $z_j=e^{i\theta_j}\in\mathbb S^1$ with $\theta_1-\theta_2\notin\mathbb Q\pi$, and let $(t_n)_{n\in\mathbb N}$ be a strictly increasing sequence in $(T_1,T)$. Let $u_j$ be the solution to \eqref{eq-gov} with source $F_j=\sigma_jf_j$, $j=1,2$. Then \eqref{eq-obs1} implies that  $F_1=F_2$.
\end{theorem}

Theorem \ref{thm-full} extends the analysis of Theorem \ref{thm-spatial} by providing, for what appears to be the first time, a full answer to problem \textbf{(ISP)} when the fractional order $\alpha$ is irrational. Its proof is carried out by combining the arguments of Theorem \ref{thm-spatial} with complex analysis techniques that allow us to remove classical limitations in inverse source problems when the fractional order $\alpha$ is irrational. This extends the continuous and one-dimensional observations by \cite{DFJK,JaK} to space-time sparse data in the two-dimensional setting.

This article is organized as follows. In Section \ref{eq-obs1} we recall preliminary properties of solutions of the IBVP
\eqref{eq-gov} while in Section \ref{s2} and \ref{s3}, we prove Theorems \ref{thm-spatial} and \ref{thm-full}, respectively.

\section{Preliminary properties}
In this section, we recall properties of problem \eqref{eq-obs1}.
We denote by $A$ the operator $-\Delta$ acting on $L^2(D)$ with its domain $D(A)=H^2(D)\cap H_0^1(D)$, and by $D(A^s)$, $s>0$, the domain of its fractional power. We equip also $L^2(D)$ with the scalar product $\langle \cdot,\cdot\rangle$
defined by $\langle f,g\rangle=\int_Df\overline{g}{\rm d}x$.

Recall that the eigenvalues of $A$, listed without multiplicity, form a strictly increasing unbounded sequence $(\lambda_n)_{n\in\mathbb N}$ of positive numbers, and the multiplicity $p_n$ of $\lambda_n$ satisfies $p_n\in\{1,2\}$ (see, e.g., \cite{li2020unique}). Each eigenspace admits an orthonormal basis $\{\varphi_{n,1},\ldots,\varphi_{n,p_n}\}$ of eigenfunctions, which in polar coordinates $(r,\theta)\in(0,1]\times(-\pi,\pi)$ can be described as follows. If $p_n=1$, then
$
\varphi_{n,1}(r,\theta)=\omega_n J_0(\lambda_n^{\frac12} r)
$, $(r,\theta)\in(0,1)\times (-\pi,\pi)$;
if $p_n=2$, then
$
\varphi_{n,1}(r,\theta)=\omega_n J_{m(n)}(\lambda_n^{\frac12} r)e^{im(n)\theta}$, $ \varphi_{n,2}(r,\theta)=\omega_n J_{m(n)}(\lambda_n^{\frac12} r)e^{-im(n)\theta}
$,
where  $m(n)\in\mathbb N:=\{1,2,\ldots\}$, $J_m$ denotes the Bessel function of the first kind of order $m$, $\lambda_n^{\frac12}$ is a positive zero of $J_{m(n)}$, and
$
\omega_n := \pi^{-\frac12}\,|J_{m(n)+1}(\lambda_n^{\frac12})|^{-1},
$
so that $\|\varphi_{n,k}\|_{L^2(D)}=1$, $k=1,\ldots,p_n$, (recall that $\int_0^1 J_m(zr)^2 r\,\mathrm{d}r=\tfrac12 J_{m+1}(z)^2$ for every positive zero $\zeta$ of $J_m$). Using the recurrence $J_m'(z)=-J_{m+1}(z)+\tfrac{m}{z}J_m(z)$, we have
\begin{equation}\label{eq-bdry}
\partial_\nu\varphi_{n,k}(z_\ell)=-\pi^{-\frac12}\lambda_n^{\frac12}\varepsilon_n e^{\pm im(n)\theta_\ell},\quad z_\ell=e^{i\theta_\ell},\ k=1,\ldots,p_n,
\end{equation}
with $\varepsilon_n:=\mathrm{sign}(J_{m(n)+1}(\lambda_n^{\frac12}))\in\{-1,1\}$ and the sign $+$ for $k=1$, $-$ for $k=2$ (for $p_n=1$ the exponential factor is understood as $1$). In particular, $\partial_\nu\varphi_{n,k}(z_\ell)\neq0$. This relation is central to the analysis below.

Now fix $z_\ell=e^{i\theta_\ell}$, $\ell=1,2$, and set 
\begin{equation}\label{h} h_\ell(t)
= \sum_{n=1}^\infty \sum_{k=1}^{p_n} \Big(\int_0^{\min(t,T)} (t-s)^{\alpha-1} E_{\alpha,\alpha}(-\lambda_n(t-s)^\alpha)\sigma(s)\,\mathrm{d}s\Big)\langle f,\varphi_{n,k}\rangle\,\partial_\nu\varphi_{n,k}(z_\ell),\  \ell=1,2,\ t>0.
\end{equation}
Then we have the following representation of the flux.
\begin{proposition}\label{p1} Let $\sigma\in L^2(0,T)$ and $f\in H^s(\Omega)$, $s\in(0,\frac12)$. Then problem \eqref{eq-obs1} admits a unique solution $u\in L^2(0,T;C^1(\overline{D}))$ satisfying, for a.e. $t\in(0,T)$, the following boundary representation 
\begin{equation}\label{p1a}\partial_\nu u(z_\ell,t)
= h_\ell(t),\quad \ell=1,2,
\end{equation}
where the sequence converges in the sense of $L^2(0,T)$. Moreover, for $\ell=1,2$, the map $h_\ell$ is analytic on $(T_1,+\infty)$.
\end{proposition}
\begin{proof} 
First, in view of \cite[Chapter~1,~Theorem~11.4]{LM1} and \cite[Chapter~1,~Theorem~11.6]{LM1}, we have $f\in D(A^{\frac s2})$. Moreover, following \cite[Theorem 2.2]{sakamoto2011initial}, the IBVP \eqref{eq-gov} with $f$ replaced by $A^{\frac s2}f$
admits a unique solution $v\in L^2(0,T;D(A))$ taking the form
$$v(\cdot,t)=\sum_{n=1}^\infty \sum_{k=1}^{p_n} \Big(\int_0^t (t-s)^{\alpha-1} E_{\alpha,\alpha}(-\lambda_n(t-s)^\alpha)\sigma(s)\,\mathrm{d}s\Big)\langle A^{\frac s2}f,\varphi_{n,k}\rangle\,\varphi_{n,k},\quad t\in(0,T),$$
where the sequence converges in $L^2(0,T;D(A))$. Also, one can check that $u=A^{-\frac s2}v$ which implies that $u\in L^2(0,T;D(A^{1+\frac s2}))$ and it takes the form
\begin{equation}\label{p1b}u(\cdot,t)=\sum_{n=1}^\infty \sum_{k=1}^{p_n} \Big(\int_0^t (t-s)^{\alpha-1} E_{\alpha,\alpha}(-\lambda_n(t-s)^\alpha)\sigma(s)\,\mathrm{d}s\Big)\langle f,\varphi_{n,k}\rangle\,\varphi_{n,k},\quad t\in(0,T),\end{equation}
where the sequence  converges in $L^2(0,T;D(A^{1+\frac s2}))$. By the continuous embedding $D(A^{1+\frac s2})\subset H^{2+s}(D)\subset C^1(\overline{D})$, we deduce that $u\in L^2(0,T;C^1(\overline{D}))$ and the sequence in  \eqref{p1b} converges in $L^2(0,T;C^1(\overline{D}))$. Then, the representation \eqref{p1a} can be easily deduced from this last statement.
By combining the above argumentation with \cite[Proposition 2.2]{JK1}, one can check that the map $w$ defined on $D\times (0,+\infty)$ by
$$w(\cdot,t)= \sum_{n=1}^\infty \sum_{k=1}^{p_n} \Big(\int_0^{\min(t,T)} (t-s)^{\alpha-1} E_{\alpha,\alpha}(-\lambda_n(t-s)^\alpha)\sigma(s)\,\mathrm{d}s\Big)\langle f,\varphi_{n,k}\rangle\,\varphi_{n,k},\quad t\in(0,+\infty)$$
is analytic with respect to $t\in(T_1,+\infty)$ as a map taking values in $D(A^{1+\frac s2})\subset H^{2+s}(D)\subset C^1(\overline{D})$. Therefore, the last statement of the proposition follows directly from the identity $h_\ell(t)=\partial_\nu w(z_\ell,t)$, $ t>0, \ell=1,2$.
\end{proof}

\section{Proof of Theorem \ref{thm-spatial}}\label{s2}
Fix $\ell\in\{1,2\}$, extend $\sigma$ by zero to $(0,\infty)$, and set $u:=u_1-u_2$ and $f:=f_1-f_2$. Let $h_\ell$ be defined by \eqref{h}. By Proposition \ref{p1}, 
$h_\ell(t)=\partial_\nu u(z_\ell, t)$, $t\in(0,T)$. In light of  \eqref{eq-obs1}, we have
$h_\ell(t_n)=0$, $n\in\mathbb N$. This, Proposition \ref{p1} and the isolated zero theorem imply that supp$(h_\ell)\subset[0,T_1]$. Hence, the Laplace transform $\widehat h_\ell$ is analytic and well-defined on $\mathbb C$ by
$\widehat h_\ell(p):=\int_0^\infty h_\ell(t)e^{-pt}=\int_0^{T} h_\ell(t)e^{-pt}\,\mathrm{d}t,\  p\in\mathbb C$.
Moreover, using classical properties of Mittag-Leffler functions (see e.g. \cite{Po}) and the identity \eqref{h} in a similar way to \cite[Proposition 2.2]{DFJK}, one can check that, for all $p\in\{z\in\mathbb C:\ |z|<\lambda_1^{\frac1\alpha},\ \mathrm{Re}\,z>0\}$, we have
\begin{equation}\label{eq-lt}
\widehat h_\ell(p)=\widehat\sigma(p)\sum_{n=1}^\infty\sum_{k=1}^{p_n}\frac{\langle f,\varphi_{n,k}\rangle}{\lambda_n+p^\alpha}\,\partial_\nu\varphi_{n,k}(z_\ell).
\end{equation}
Since both sides of \eqref{eq-lt} are analytic in $\mathcal Q=\{p\in\mathbb C\setminus(-\infty,0]: |p|<\lambda_1^{\frac1\alpha}\}$, by the isolated zero theorem, the identity \eqref{eq-lt} is still true on $\mathcal Q$. Thus, by fixing  $r\in(0,\lambda_1^{\frac1\alpha})$ and using the continuity of $\widehat h_\ell$ in $\mathbb C$, we get
$$\begin{aligned}
    0&=\widehat h_\ell(-r)-\widehat h_\ell(-r)=\lim_{\theta\to-\pi^+}\widehat h_\ell(re^{i\theta})-\lim_{\theta\to\pi^-}\widehat h_\ell(re^{i\theta})\\
&=\widehat\sigma(-r)\Big(\sum_{n=1}^\infty\sum_{k=1}^{p_n}\frac{\langle f,\varphi_{n,k}\rangle}{\lambda_n+r^\alpha e^{-i\pi\alpha}}\,\partial_\nu\varphi_{n,k}(z_\ell)
-\sum_{n=1}^\infty\sum_{k=1}^{p_n}\frac{\langle f,\varphi_{n,k}\rangle}{\lambda_n+r^\alpha e^{i\pi\alpha}}\,\partial_\nu\varphi_{n,k}(z_\ell)\Big),
\end{aligned}$$
where the limits may be interchanged with the sums by the uniform convergence in $\theta\in[-\pi,\pi]$. Euler's formula then gives
$$
0=2i\,\widehat\sigma(-r)\,r^\alpha\sin(\pi\alpha)\sum_{n=1}^\infty\frac{\sum_{k=1}^{p_n}\langle f,\varphi_{n,k}\rangle \partial_\nu\varphi_{n,k}(z_\ell)}{(\lambda_n+r^\alpha e^{i\pi\alpha})(\lambda_n+r^\alpha e^{-i\pi\alpha})}.
$$
Since $\sigma\not\equiv0$, $\widehat\sigma$ is a nonzero entire function, so its zeros are isolated and there exist $0<r_1<r_2<\lambda_1^{\frac1\alpha}$ such that $\widehat\sigma(-r)\neq0$ for all $r\in(r_1,r_2)$. Since $\sin(\pi\alpha)\neq0$, we deduce
\begin{equation}\label{eq-H}
H(s):=\sum_{n=1}^\infty\frac{\sum_{k=1}^{p_n}\langle f,\varphi_{n,k}\rangle \partial_\nu\varphi_{n,k}(z_\ell)}{(\lambda_n+se^{i\pi\alpha})(\lambda_n+se^{-i\pi\alpha})}\,=0,\quad s\in(r_1^\alpha,r_2^\alpha).
\end{equation}
The set $\mathcal O:=\mathbb C\setminus\{-\lambda_ne^{\pm i\pi\alpha}: n\in\mathbb N\}$ is connected, and moreover, $H$ extends analytically to $\mathcal O$. By the isolated zero theorem, $H\equiv0$ on $\mathcal O$, which implies that
\begin{equation}\label{eq-c}
\sum_{k=1}^{p_n}\langle f,\varphi_{n,k}\rangle\,\partial_\nu\varphi_{n,k}(z_\ell)=0,\quad n\in\mathbb N,\ \ell=1,2.
\end{equation}
For $p_n=1$, \eqref{eq-c} and \eqref{eq-bdry} yield $\langle f,\varphi_{n,1}\rangle=0$. For $p_n=2$, dividing \eqref{eq-c} for $\ell=1,2$ by the common nonzero factor $-\pi^{-\frac12}\lambda_n^{\frac12}\varepsilon_n$ gives the linear system $$e^{im(n)\theta_\ell}\langle f,\varphi_{n,1}\rangle+e^{-im(n)\theta_\ell}\langle f,\varphi_{n,2}\rangle=0,\quad \ell=1,2,
$$
whose determinant $e^{im(n)(\theta_1-\theta_2)}-e^{-im(n)(\theta_1-\theta_2)}=2i\sin(m(n)(\theta_1-\theta_2))$ is nonzero, because $\theta_1-\theta_2\notin\mathbb Q\pi$. Hence $\langle f,\varphi_{n,k}\rangle=0$ for all $n\in\mathbb N$ and $k=1,\ldots,p_n$, and since $\{\varphi_{n,k}\}$ is an orthonormal basis of $L^2(D)$, we get $f=0$, i.e., $f_1=f_2$.

\section{Proof of Theorem \ref{thm-full}}\label{s3}

We extend $\sigma_j$ by zero to $(0,\infty)$, and for $j,\ell\in\{1,2\}$, let $h_{j,\ell}$ be given by \eqref{h}, with $\sigma=\sigma_j$ and $f=f_j$, and set $h_\ell:=h_{1,\ell}-h_{2,\ell}$. In view of Proposition \ref{p1}, we have $h_{j,\ell}(t)=\partial_\nu u_j(z_\ell,t)$ for $t\in(0,T)$ and  \eqref{eq-obs1} implies that $h_\ell(t_n)=0$, $n\in\mathbb N$. In light of Proposition \ref{p1} and the isolated zero theorem, this condition implies that supp$(h_\ell)\subset[0,T_1]$ and the Laplace transform of $h_\ell$ is analytic and well defined on $\mathbb C$ by
$\widehat h_\ell(p)=\int_0^T h_\ell(t)e^{-pt}\,\mathrm{d}t$, $p\in\mathbb C$. 
Similarly to Theorem \ref{thm-spatial}, by taking the Laplace transform termwise in \eqref{h},  for all $p\in\{z\in\mathbb C:\ |z|<\lambda_1^{\frac1\alpha},\ \mathrm{Re}\,z>0\}$, we have
$$\widehat h_\ell(p)=\sum_{n=1}^\infty\sum_{k=1}^{p_n}\frac{\widehat\sigma_1(p)\langle f_1,\varphi_{n,k}\rangle-\widehat\sigma_2(p)\langle f_2,\varphi_{n,k}\rangle}{\lambda_n+p^\alpha}\,\partial_\nu\varphi_{n,k}(z_\ell).
$$
Then, repeating the argumentation of Theorem \ref{thm-spatial}, we can extend this identity to  $\{p\in\mathbb C\setminus(-\infty,0]: |p|<\lambda_1^{\frac1\alpha}\}$ and deduce that, for all $r\in(0,\lambda_1^{\frac1\alpha})$, we have
$$\begin{aligned}
0&=\lim_{\theta\to-\pi^+}\widehat h_\ell(re^{i\theta}) - \lim_{\theta\to\pi-}\widehat h_\ell(re^{i\theta}) = \sum_{j=1}^2(-1)^{j+1}\sum_{n=1}^\infty\sum_{k=1}^{p_n}\frac{\widehat\sigma_1(-r)\langle f_1,\varphi_{n,k}\rangle-\widehat\sigma_2(-r)\langle f_2,\varphi_{n,k}\rangle}{\lambda_n+r^\alpha e^{i(-1)^{j}\alpha\pi}}\,\partial_\nu\varphi_{n,k}(z_\ell).\end{aligned}
$$
With the shorthand notation
$$
\psi_{j,\ell}(z):=\sum_{n=1}^\infty\sum_{k=1}^{p_n}\frac{\langle f_j,\varphi_{n,k}\rangle}{(\lambda_n+ze^{-i\alpha\pi})(\lambda_n+ze^{i\alpha\pi})}\,\partial_\nu\varphi_{n,k}(z_\ell),\quad z\in\mathcal O_1:=\mathbb C\setminus\{-\lambda_ne^{\pm i\alpha\pi}: n\in\mathbb N\},
$$
the last identity becomes
$$r^\alpha(e^{i\alpha\pi}-e^{-i\alpha\pi})\,\big(\widehat\sigma_1(-r)\,\psi_{1,\ell}(r^\alpha)-\widehat\sigma_2(-r)\,\psi_{2,\ell}(r^\alpha)\big)=0.$$ Since $r^\alpha(e^{i\alpha\pi}-e^{-i\alpha\pi})\neq0$, we get
\begin{equation}\label{eq-psi2}
\widehat\sigma_1(-r)\,\psi_{1,\ell}(r^\alpha)-\widehat\sigma_2(-r)\,\psi_{2,\ell}(r^\alpha)=0,\quad r\in(0,\lambda_1^{\frac1\alpha}).
\end{equation}
It is clear that each $\psi_{j,\ell}$ is holomorphic on $\mathcal O_1$, and $\mathcal O_1\setminus(-\infty,0]$ is connected. The map
$
\mathcal R:z\mapsto \widehat\sigma_1(-z^{\frac1\alpha})\,\psi_{1,\ell}(z)-\widehat\sigma_2(-z^{\frac1\alpha})\,\psi_{2,\ell}(z)
$
is holomorphic on $\mathcal O_1\setminus(-\infty,0]$, and by \eqref{eq-psi2} it vanishes on $(0,\lambda_1)$; hence $\mathcal R=0$ on $\mathcal O_1\setminus(-\infty,0]$. Taking $z=re^{i\theta}$, $r\in(0,\lambda_1^{\frac1\alpha})$, and letting $\theta\to\pi^-$ and $\theta\to-\pi^+$,  gives
\begin{equation}\label{eq-sys}
\widehat\sigma_1(-r^{\frac1\alpha}e^{\pm\frac{ i\pi}{\alpha}})\,\psi_{1,\ell}(-r)-\widehat\sigma_2(-r^{\frac1\alpha}e^{\pm\frac{ i\pi}{\alpha}})\,\psi_{2,\ell}(-r)=0,\quad r\in(0,\lambda_1^{\frac1\alpha}).
\end{equation}

If $f_1=f_2\equiv 0$, then $F_1=\sigma_1f_1=0=\sigma_2f_2=F_2$, and the desired conclusion follows immediately. Hence, we may assume from now on that $(f_1,f_2)\neq(0,0)$.

We claim that for at least one $\ell\in\{1,2\}$, the pair $(\psi_{1,\ell},\psi_{2,\ell})$ is not identically zero. Indeed, if $\psi_{j,\ell}\equiv0$ for both $\ell\in\{1,2\}$ and all $j$, then multiplying $\psi_{j,\ell}$ by $(z+\lambda_ne^{i\alpha\pi})(z+\lambda_ne^{-i\alpha\pi})$ and sending $z\to-\lambda_ne^{i\alpha\pi}$ yields $\sum_{k=1}^{p_n}\langle f_j,\varphi_{n,k}\rangle\,\partial_\nu\varphi_{n,k}(z_\ell)=0$ for all $n,j,\ell$; the determinant argument at the end of the proof of Theorem~\ref{thm-spatial} then forces $f_j\equiv0$, which contradicts $(f_1,f_2)\neq(0,0)$. We fix such an $\ell$ from now on. Since $\psi_{1,\ell}$ and $\psi_{2,\ell}$ are holomorphic on $\mathcal O_1\supset(-\lambda_1,0)$ and are not both identically zero, their common zeros on $(-\lambda_1,0)$ are isolated; hence, shrinking if necessary, there exists $r_0\in(0,\lambda_1^{\frac1\alpha})$ such that
$
(\psi_{1,\ell}(-r),\psi_{2,\ell}(-r))\neq(0,0),\ r\in(0,r_0)
$.
By \eqref{eq-sys}, for $r\in(0,r_0)$ the nonzero vector $(\psi_{1,\ell}(-r),-\psi_{2,\ell}(-r))$ lies in the kernel of the matrix
$$
M_r:=\begin{pmatrix}
\widehat\sigma_1(-r^{\frac1\alpha}e^{\frac{i\pi}{\alpha}}) & \widehat\sigma_2(-r^{\frac1\alpha}e^{\frac{i\pi}{\alpha}}) \\[2mm]
\widehat\sigma_1(-r^{\frac1\alpha}e^{-\frac{i\pi}{\alpha}}) & \widehat\sigma_2(-r^{\frac1\alpha}e^{-\frac{i\pi}{\alpha}})
\end{pmatrix},
$$
so $\det M_r=0$, i.e., with $s:=r^{\frac1\alpha}\in(0,r_0^{\frac1\alpha})$,
\begin{equation}\label{eq-det}
\widehat\sigma_1(-se^{\frac{i\pi}{\alpha}})\,\widehat\sigma_2(-se^{-\frac{i\pi}{\alpha}})-\widehat\sigma_1(-se^{-\frac{i\pi}{\alpha}})\,\widehat\sigma_2(-se^{\frac{i\pi}{\alpha}})=0,\quad s\in(0,r_0^{\frac1\alpha}).
\end{equation}
Since $\sigma_j\not\equiv0$, $\widehat\sigma_j$ is a nonzero entire function, and its zeros are isolated. Hence, there exists $r_1\in(0,r_0^{\frac1\alpha})$ such that $\widehat\sigma_j(z)\neq0$ for all $z\in\mathbb C$ with $0<|z|<r_1$, $j=1,2$. On the disk $ \mathbb D_{r_1}:=\{z\in\mathbb C: |z|<r_1\}$ the quotient $\widehat\sigma_1/\widehat\sigma_2$ is meromorphic with (at most) a single singularity at $0$; upon possibly interchanging the roles of the indices $1$ and $2$, we may assume without loss of generality that there exists $g$ holomorphic on $\mathbb D_{r_1}$ such that
$
\widehat\sigma_1(z)=g(z)\,\widehat\sigma_2(z)$, $z\in \mathbb D_{r_1}$.
Substituting this into \eqref{eq-det} yields
$$
\widehat\sigma_2(-se^{\frac{i\pi}{\alpha}})\,\widehat\sigma_2(-se^{-\frac{i\pi}{\alpha}})\,\big(g(-se^{\frac{i\pi}{\alpha}})-g(-se^{-\frac{i\pi}{\alpha}})\big)=0,\quad s\in(0,r_1).
$$
Since $|-se^{\pm\frac{i\pi}{\alpha}}|=s\in(0,r_1)$, the two factors $\widehat\sigma_2(-se^{\pm\frac{i\pi}{\alpha}})$ are nonzero. Consequently,
$$
g(we^{\frac{i\pi}{\alpha}})=g(we^{-\frac{i\pi}{\alpha}}),\quad w\in(-r_1,0)
.$$
Thus, the function $z\mapsto g(ze^{\frac{i\pi}{\alpha}})-g(ze^{-\frac{i\pi}{\alpha}})$ is holomorphic on $\mathbb D_{r_1}$ and vanishes on $(-r_1,0)$, so by the isolated zero theorem it vanishes on all of $\mathbb D_{r_1}$. Differentiating $m$ times and evaluating at $z=0$ gives
$
e^{\frac{im\pi}{\alpha}}g^{(m)}(0)=e^{-\frac{im\pi}{\alpha}}g^{(m)}(0)$, $m\ge1$.
Since $\alpha$ is irrational, $e^{\frac{2im\pi}{\alpha}}\neq1$ for every $m\ge1$ (otherwise $m/\alpha\in\mathbb Z$, forcing $\alpha\in\mathbb Q$); consequently $g^{(m)}(0)=0$ for all $m\ge1$, so $g$ is constant on $\mathbb D_{r_1}$, say $g\equiv C$. Note that $C\neq0$ (otherwise $\widehat\sigma_1\equiv0$ on $\mathbb D_{r_1}$ and hence $\sigma_1\equiv0$, a contradiction). Thus
$
\widehat\sigma_1(p)=C\,\widehat\sigma_2(p)$, for $ p\in(0,r_1)
$.
Both $\widehat\sigma_1$ and $\widehat\sigma_2$ are entire, so the identity theorem gives $\widehat\sigma_1\equiv C\widehat\sigma_2$ on $\mathbb C$, and the injectivity of the Laplace transform implies $\sigma_1=C\sigma_2$ a.e.~on $(0,\infty)$, in particular on $(0,T)$.

Therefore $u_1$ solves \eqref{eq-gov} with the source $\sigma_2(t)\widetilde f_1(x)$, where $\widetilde f_1:=Cf_1\in D(A^s)$. Then the application of Theorem~\ref{thm-spatial} with $\sigma=\sigma_2$ and the pair $(\widetilde f_1,f_2)$ (which satisfy all its hypotheses) yields $f_2=\widetilde f_1=Cf_1$. Hence
$
F_1(x,t)=\sigma_1(t)f_1(x)=C\sigma_2(t)f_1(x)=\sigma_2(t)f_2(x)=F_2(x,t),\  (x,t)\in D \times (0,T)
$,
which completes the proof.

\bibliographystyle{abbrv}
\bibliography{ref.bib}

\end{document}